\documentclass[10pt]{amsart}
\usepackage{amsmath}
\usepackage{amsthm}
\usepackage{color}
\usepackage{amscd}
\usepackage{amsfonts}
\usepackage{graphicx}
\usepackage{epsfig}
\usepackage{amssymb,latexsym}
\usepackage{bm}

\newcommand{\dv}{\operatorname{div}}
\def\bbu{{\boldsymbol{u}}}
\def\bbx{{\boldsymbol{x}}}
\def\by{\mathbf{y}}
\def\R{\mathbb{R}}
\def\bcero{\mathbf 0}
\numberwithin{equation}{section}

\newtheorem{theorem}{Theorem}[section]
\newtheorem{corollary}[theorem]{Corollary}
\newtheorem{definition}{Definition}[section]
\newtheorem{lemma}[theorem]{Lemma} 
\newtheorem{proposition}[theorem]{Proposition}

\begin{document}
\title[Optimal control. Existence and relaxation]{Existence and relaxation for optimal control governed by steady, quasilinear PDEs}
\author{Pablo Pedregal$^\dagger$}
\date{} 
\thanks{INEI, U. de Castilla-La Mancha, 13071 Ciudad Real, SPAIN. Supported by grants 
PID2020-116207GB-I00, and  SBPLY/19/180501/000110. Data sharing not applicable to this article as no datasets were generated or analysed during the current study.}
\thanks{$\dagger$ Corresponding author: pablo.pedregal@uclm.es}
\begin{abstract}
We focus on optimal control problems governed by elliptic, quasilinear PDEs. Though there are various examples of such problems in the literature, we make an attempt at describing some general principles by dealing with three basic situations. In the first one, we assume that the state equation is variational; the second one focuses on a non-variational, monotone operator as state equation; finally, we add a non-linear term off the divergence part of the equation. In the first two cases, existence of optimal solutions can be established under suitable sets of assumptions, while relaxation is required for the third situation. Concerning the cost functional, and though more general examples can be dealt with, we will take a typical case consisting of two terms: one depending on the state, and another one of the form of a typical Thychonov regularization. Optimality, especially for the last situation, will be addressed in a forthcoming work. 
\end{abstract}
\maketitle
\section{Introduction}
We would like to address the basic issues of existence of optimal solutions and optimality conditions for optimal control problems under state equations of the general structure
\begin{equation}\label{estado}
-\dv[A(\bbx, y_u(\bbx), \nabla y_u(\bbx))]+a(\bbx, y_u(\bbx), \nabla y_u(\bbx))=f(u(\bbx))\hbox{ in }\Omega,
\end{equation}
and some appropriate boundary condition around $\partial\Omega$ for $y_u$ in such a way that the operation $u\mapsto y_u$ is well-defined. Here
$$
A(\bbx, y, \by):\Omega\times\R\times\R^N\to\R^N,\quad a(\bbx, y, \by):\Omega\times\R\times\R^N\to\R,
$$
are functions that are assumed continuous, or even smooth, in variables $(y, \by)$ and measurable in $\bbx$, while
$f(u):\R\to\R$ is a continuous function. Since our interest is not to examine results under the most general hypotheses on the domain $\Omega$, we will assume it to be as regular as we may need it to be. 
The variable $u$ represents the control while $y_u$ stands for its associated state. 
For the cost functional, we will put
$$
E(u)=\int_\Omega F(\bbx, y_u(\bbx), u(\bbx), \nabla y_u(\bbx), \nabla u(\bbx))\,d\bbx
$$
for an integrand
$$
F(\bbx, y, u, \by, \bbu): \Omega\times\R\times\R\times\R^N\times\R^N\to\R
$$
which will be assumed with the necessary hypotheses in each particular situation considered. 

One could envision even more general problems where the interaction of state and control both in the state equation and the cost functional runs at a deeper, more general level beyond the situations to be explored here. But such degree of generality leads to a far too general framework in which the structure of the state equation \eqref{estado} might be completely lost. Keep in mind the well-known case of optimal control problems where the control acts on the coefficients of a linear elliptic equation, and the fundamental field of homogenization (\cite{murattartar}). 

For the sake of definiteness, and to stress what our main contribution is, we will focus on state equations of the simplified form
\begin{equation}\label{estadosimp}
-\dv[A(\nabla y_u(\bbx))]+a(\nabla y_u(\bbx))=f(u(\bbx))\hbox{ in }\Omega,
\end{equation}
overlooking explicit dependence on pairs $(\bbx, y_u)$. These more general situations can easily be adapted.
To be specific, with the target in mind of learning how to deal with non-linearities, we will treat three situations of increasing complexity, as regards the state equation and its structure:
\begin{enumerate}
\item State equations of variational nature with $a\equiv0$.
\item Non-linear, monotone state equation with $a\equiv0$.
\item Quasilinear state equation: $a$, non-vanishing.
\end{enumerate}


In the first situation, we assume that the state equation \eqref{estadosimp} is variational, i.e. it is the optimality equation with respect to $y$ for a cost functional
$$
I(y, u)=\int_\Omega [W(\nabla y(\bbx))+f(u(\bbx))y(\bbx)]\,d\bbx
$$
where the integrands
$$
W(\by):\R^N\to\R,\quad f(u):\R\to\R,
$$
are uniformly of quadratic growth and convex, and continuous, respectively, in such a way that the state equation \eqref{estadosimp} becomes
$$
-\dv[\nabla_\by W(\nabla y_u(\bbx))]+f(u(\bbx)))=0\hbox{ in }\Omega,
$$
and
$$
c(|\by|^2-1)\le W(\by)\le C(|\by|^2+1),\quad 0<c<C.
$$
Compared to our general model problem \eqref{estadosimp}, we see that the coefficient $a\equiv0$ and the divergence part is variational.

The second case corresponds to a non-variational, non-linear, monotone equation. This time, for the state equation in \eqref{estadosimp} we have the crucial monotonicity property
\begin{equation}\label{monotonia1}
(A(\by_1)-A(\by_2))\cdot(\by_1-\by_2)\ge0;
\end{equation}
moreover, the condition
\begin{equation}\label{monotonia2}
(A(\by_1)-A(\by_2))\cdot(\by_1-\by_2)=0
\end{equation}
implies $\by_1=\by_2$. The map $A(\by):\R^N\to\R^N$ is assumed continuous and of linear growth
\begin{equation}\label{crecimiento}
|A(\by)|\le C(|\by|+1),\quad C>0.
\end{equation}
Coefficient $a$ still vanishes identically. 

In the final model, we will take a state equation of the form \eqref{estadosimp}, where this time the function $a$ is non-vanishing, $a(\bcero)=0$, and it is Lipschitz. In particular, we will have
$$
|a(\by)|\le C|\by|,\quad C>0.
$$
If the mapping $A$ is elliptic, then it is well-known that the corresponding state equation will have a unique solution for each feasible control function $u$, provided, for instance, that the right-hand side $f(u)$ in \eqref{estadosimp} is a uniformly bounded, continuous function. Since in this third situation, we will mainly be concerned with the effect of the non-linearity $a(\by)$, we will simplify the main part of the equation as much as possible, and will write $A(\by)=\by$, so that it becomes the Laplace operator, and our state equation will be
$$
-\Delta y_u+a(\nabla y_u)+by_u=f(u)\hbox{ in }\Omega,\quad y_u\in H^2(\Omega)\cap H^1_0(\Omega),
$$
for a scalar $b$ sufficiently large to ensure existence and uniqueness of the solution $y_u$. 

Concerning the cost functional, again for the sake of definiteness, we will always take
$$
E(u)=\int_\Omega F(y_u(\bbx))\,d\bbx+\frac M2\int_\Omega|\nabla u(\bbx)|^2\,d\bbx,
$$
where $F(y)$ is a continuous real function, that could also depend on $\bbx$, and $M>0$ is a constant. Eventually, one could also take 
$$
E(u)=\int_\Omega F(y_u(\bbx))\,d\bbx+\frac M2\int_\Omega|u(\bbx)|^2\,d\bbx,
$$
taking this time control variables $u$ in $L^2(\Omega)$ instead of in $H^1(\Omega)$. We will comment on how this change may affect our results.

We will be able to show existence of optimal solutions in the first two situations, when $a\equiv0$: Theorem \ref{existencia}, Corollary \ref{cor}, and Theorem \ref{existencia2}; 
however, when this coefficient $a$ is non-vanishing,  non-existence is expected, in general, and the corresponding optimization problem is in need of relaxation (Theorem \ref{relajaciont}). In addition to recalling some fundamental results from the theory of Young measures (Theorems \ref{basico} and \ref{importante}) and some corollaries adapted to our needs in this work (Corollaries \ref{modificar} and \ref{aqui}), our main analytical concept in this framework is that of measure-valued solution of a PDE. This is a quite common concept in the field of conservation laws or in the area of Navier-Stokes equations, but, to the best of our knowledge, have not received much attention outside those  fields. The literature on this topic is rapidly expanding though we will only cite the pioneering book \cite{necas}. Since Young measures is a common tool in the analysis of non-convex variational problems (\cite{balder}, \cite{muller}, \cite{pedregalbook}, \cite{rindler}, \cite{roubicek}, \cite{valadier}), it is not surprising that they play a central role for non-linear optimal control problems as well (\cite{roubicekbook}).  

\begin{definition}
A family of probability measures $\nu=\{\nu_\bbx\}_{\bbx\in\Omega}$, supported in $\R^N$ complying with the integrability condition
$$
\int_\Omega\int_{\R^N}|\lambda|^2\,d\nu_\bbx(\lambda)\,d\bbx<\infty,
$$
is a measure-valued solution of the PDE
$$
-\dv[A(\nabla y(\bbx))]+a(\nabla y(\bbx))+f(\bbx)=0\hbox{ in }\Omega,\quad y\in H^1_0(\Omega),
$$ 
with $A$, and $a$, continuous, and
$$
|A(\by)|\le C(|\by|+1),\quad |a(\by)|\le C|\by|,\quad C>0,
$$
if
$$
-\dv[\overline A(\bbx)]+\overline a(\bbx)+f(\bbx)=0\hbox{ in }\Omega, 
$$
where
\begin{gather}
\nabla y(\bbx)=\int_{\R^N}\lambda\,d\nu_\bbx(\lambda),\quad y\in H^2(\Omega)\cap H^1_0(\Omega),\nonumber\\ 
\overline A(\bbx)\equiv\int_{\R^N}A(\lambda)\,d\nu_\bbx(\lambda),\quad \overline a(\bbx)\equiv\int_{\R^N}a(\lambda)\,d\nu_\bbx(\lambda).\nonumber
\end{gather}
\end{definition}
Relaxation for the last situation examined will require this kind of generalized state equation with measure-valued solutions (Theorem \ref{relajaciont}). 

Optimal control problems for non-linear PDEs have been extensively studied, especially the semilinear case. The bibliography for this case is abundant covering a rich spectrum of issues and possibilities \cite{trolt}. See also \cite{roubicekbook}, \cite{roubicek} for a finer, more sophisticated treatment. 
The quasilinear case has also been addressed (\cite{casas4}, \cite{casas5}) even for dynamic problems \cite{casas2}, but for state equations having a linear structure, with coefficients possibly depending on the state, on the gradient of the state. The special case of the $p$-Laplace equation has received some attention too \cite{casas3}, \cite{varios}. 

Our methods can be applied, without essential changes, to more general situations, especially in the first two cases. We will treat and comment on some of them as we move on to deal with each of the situations indicated above.


\section{State equations with variational structure}\label{variacional}
Although, as indicated in the Introduction, more general situations can be treated, for the sake of simplicity we will focus on the following optimal control problem
\begin{equation}\label{optcontprob}
\hbox{Minimize in }u\in H^1(\Omega):\quad E(u)=\int_\Omega F(y_u(\bbx))\,d\bbx+\frac M2\int_\Omega|\nabla u(\bbx)|^2\,d\bbx
\end{equation}
subject to
\begin{gather}
I(y_u, u)=\min_{y\in H^1_0(\Omega)} I(y, u)\label{estadovar}\\
I(y, u)=\int_\Omega [W(\nabla y(\bbx), u(\bbx))+f(\bbx)y(\bbx)]\,d\bbx.\nonumber
\end{gather}
Assumptions on the various elements defining this problem are:
\begin{enumerate}
\item Integrand $F$ is a continuous, bounded real function. If $N>2$, it can only be assumed such that 
$$
-C\le F(y)\le C(|y|^{2N/(N-2)}+1),\quad C>0.
$$
\item Integrand $W(\by, u)$ is strictly convex in $\by$ for every $u$, and 
$$
c(|\by|^2-1)\le W(\by, u)\le C(|\by|^2+1),\quad 0<c<C.
$$
\item $M>0$ is a constant, and $f\in L^2(\Omega)$.
\end{enumerate}
Under these assumptions, it is elementary to check that there is a unique $y_u$, minimizer in \eqref{estadovar}, and the optimization problem is well-defined. 

\begin{theorem}\label{existencia}
Under the hypotheses just described, there is an optimal solution $u_0$ for our optimal control problem \eqref{optcontprob}. 
\end{theorem}
The proof of this result follows that utilized in \cite{pedregalvarios} in a much more involved situation of non-linear systems of PDEs. It is however worthwhile to write it in this context, so as to bring attention to it. 
\begin{proof}
Let $\{u_j\}$ be a minimizing sequence for our problem, with corresponding states $\{y_j\}$. The bounds assumed on the various ingredients of our problem enables us to conclude that the first sequence is uniformly bounded in $H^1(\Omega)$ (the integrand $F$ is bounded from below by some constant), and so is the second. Consequently, there are, after taking suitable subsequence not relabeled, weak limits $u$ and $y$, respectively. The issue is if they are related through the minimization problem \eqref{estadovar}, i.e. if $y=y_u$ for the limit $u$. This is indeed so. 

Let $z$ be an arbitrary function in $H^1_0(\Omega)$. For each $j$, we should have
$$
I(y_j, u_j)\le I(z, u_j)
$$
that is
$$
\int_\Omega [W(\nabla y_j(\bbx), u_j(\bbx))+f(\bbx)y(\bbx)]\,d\bbx\le \int_\Omega [W(\nabla z(\bbx), u_j(\bbx))+f(\bbx)z(\bbx)]\,d\bbx,
$$
because $y_j$ is the unique minimizer corresponding to $u_j$. We would like to take limits in the previous inequality. To this end, note that $u_j\to u$ strongly in $L^2(\Omega)$, and that $W$ is convex in $\by$, for fixed $u$. Thus by weak lower semicontinuity
\begin{equation}\label{wls}
I(y, u)\le \lim_j I(y_j, u_j).
\end{equation}
Notice that no convexity of $W$ is necessary in the variable $u$ because the convergence $u_j\to u$ is strong, and not just weak. Hence
\begin{align}
I(y, u)\le &\lim_j \int_\Omega [W(\nabla z(\bbx), u_j(\bbx))+f(\bbx)z(\bbx)]\,d\bbx\nonumber\\
=&\int_\Omega [W(\nabla z(\bbx), u(\bbx))+f(\bbx)z(\bbx)]\,d\bbx,\label{igualdad}
\end{align}
again by the same strong convergence $u_j\to u$, and our bounds on $W$. The resulting inequality for arbitrary $z\in H^1_0(\Omega)$, implies that the limit $y$ is indeed $y_u$ for $u$ the limit of the initial minimizing sequence for the problem. 

Once we have this crucial information at our disposal, and realizing that, once again, $y_j\to y\equiv y_u$ in $L^2(\Omega)$ strongly, and that $E(u_j)\searrow m$, the value of the infimum of the control problem, by weak lower semicontinuity of $E$ with respect to $\{u_j\}$, 
\begin{align}
m\le &E(u)=\int_\Omega F(y_u(\bbx))\,d\bbx+\frac M2\int_\Omega|\nabla u(\bbx)|^2\,d\bbx\nonumber\\
\le&\lim_j \int_\Omega F(y_j(\bbx))\,d\bbx+\frac M2\int_\Omega|\nabla u_j(\bbx)|^2\,d\bbx\label{desigualdad}\\
=&\lim_j E(u_j)\nonumber\\
=&m.\nonumber
\end{align}
The limit pair $(y, u)$ is then feasible, and becomes an optimal solution. 
\end{proof}
The situation in which the cost functional does not depend, in an explicit, coercive form, in the gradient of the control, 
$$
E(u)=\int_\Omega F(y_u(\bbx), u(\bbx))\,d\bbx
$$
is drastically distinct, because we can no longer rely on the strong convergence of a minimizing sequence $\{u_j\}$, but only weak is guaranteed (under coercivity of $F$ with respect to $u$). The most demanding issue for the proof of Theorem \ref{existencia} to be valid in this other case, concerns equality \eqref{igualdad}. This forces, under weak convergence of $u_j\rightharpoonup u$, that the dependence of $W$ in $u$ be affine. In addition, inequality \eqref{wls} requires $W$ to be fully convex in all of its variables, and inequality \eqref{desigualdad} asks for a convex dependence of $F$ in $u$. 
We are talking about the problem
\begin{equation}\label{optcontprobesp}
\hbox{Minimize in }u\in H^1(\Omega):\quad E(u)=\int_\Omega F(y_u(\bbx), u(\bbx))\,d\bbx
\end{equation}
subject to
\begin{gather}
I(y_u, u)=\min_{y\in H^1_0(\Omega)} I(y, u)\nonumber\\
I(y, u)=\int_\Omega [W(\nabla y(\bbx))+ w(y(\bbx))u(\bbx)+f(\bbx)y(\bbx)]\,d\bbx.\nonumber
\end{gather}
Note that the state equation reads
$$
-\dv[\nabla_\by W(\nabla y_u(\bbx))]+w'(y_u(\bbx))u(\bbx)+f(\bbx)=0\hbox{ in }\Omega,
$$
in case $W$ and $w$ are differentiable, and suitable bounds are assumed on such derivatives. 

The existence result requires 
assumptions to be changed to:
\begin{enumerate}
\item Integrand $F(y, u)$ is a continuous, real function, convex in $u$ for each fixed $y$ such that 
$$
C(|u|^2-1)\le F(y, u),\quad C>0.
$$
\item Integrand $W(\by)$ is strictly convex, and 
$$
c(|\by|^2-1)\le W(\by)\le C(|\by|^2+1),\quad 0<c<C.
$$
\item $f\in L^2(\Omega)$, and the real function $w(y)$ is continuous. 
\end{enumerate}

\begin{corollary}\label{cor}
Under this new set of assumptions, problem \eqref{optcontprobesp} admits optimal solutions. 
\end{corollary}
\begin{proof}
The proof follows exactly along the lines of that of Theorem \ref{existencia}. We have enforced hypotheses to ensure that inequality \eqref{wls}, equality \eqref{igualdad}, and inequality \eqref{desigualdad} still hold. Hence the outcome is similar: the existence of optimal solutions for problem \eqref{optcontprobesp}. 
\end{proof}

\section{Monotone state equations}
We now deal with the second situation in which the state equation does not have variational structure. More specifically our problem is
\begin{equation}\label{probnonvar}
\hbox{Minimize in }u\in H^1(\Omega):\quad \int_\Omega F(y_u(\bbx))\,d\bbx+\frac M2\int_\Omega|\nabla u(\bbx)|^2\,d\bbx
\end{equation}
under
\begin{equation}\label{estado2}
-\dv[A(\nabla y_u(\bbx))]=f(\bbx, u(\bbx))\hbox{ in }\Omega,\quad y_u\in H^1_0(\Omega),
\end{equation}
for a positive constant $M$, and function 
\begin{equation}\label{rhs}
f(\bbx, u):\Omega\times\R\to\R,\quad |f(\bbx, u)|\le C(|u|+1).
\end{equation}
The function $F$ is assumed non-negative and continuous. 
Concerning the central map $A(\by)$, we will assume that it is strictly monotone as expressed in \eqref{monotonia1}-\eqref{monotonia2}, more specifically
\begin{equation}\label{monotonia}
(A(\by_1)-A(\by_2))\cdot(\by_1-\by_2)\ge c|\by_1-\by_2|^2,\quad c>0,
\end{equation}
and complies with
\begin{equation}\label{monotonia3}
|A(\by)|\le C(|\by|+1),\quad A(\by)\cdot\by\ge c(|\by|^2-1),\quad C>c>0.
\end{equation}
Under these circumstances, for a feasible control $u$, there is a unique state $y_u$ solution of \eqref{estado2} (check for instance \cite{evans}). The proof of Theorem \ref{existencia} can no longer be adapted to this situation. 

An interesting family of examples of this kind is that of a perturbed variants of linear elliptic equations of the form
$$
-\dv[a\nabla y(\bbx)+g(\nabla y(\bbx))]=f(\bbx, u(\bbx))\hbox{ in }\Omega,
$$
where the coefficient $a$ could depend on $\bbx$, and
$$
a(\bbx)\ge a_0>0,
$$
while the mapping $g:\R^N\to\R^N$ is globally Lipschitz with a Lipschitz constant $L$ such that $a_0-L>0$. Without loss of generality, we can assume $g(\bcero)=\bcero$, and hence
$$
|g(\by)|\le L|\by|.
$$
Under these circumstances, the map
$$
A(\by)=a\by+g(\by)
$$
turns out to be strictly monotone \eqref{monotonia}-\eqref{monotonia3}. This is elementary to check. 

To learn about difficulties we face better, suppose $\{u_j\}$ is a bounded, minimizing sequence in $H^1(\Omega)$. It converges weakly in $H^1(\Omega)$ and strongly in $L^2(\Omega)$, after identifying a suitable non-relabeled subsequence, to some $u$. Let $\{y_j\}$ be the associated sequence of states so that
\begin{equation}\label{ecuacion}
-\dv[A(\nabla y_j)]=f(\bbx, u_j)\hbox{ in }\Omega,\quad y_j\in H^1_0(\Omega).
\end{equation}
If we use $y_j$ itself as a test function, and bearing in mind \eqref{monotonia3}, we conclude, in a standard way, that $\{y_j\}$ is uniformly bounded in $H^1_0(\Omega)$, and, as such, after a suitable non-relabeled subsequence, it converges to some $y\in H^1_0(\Omega)$. As in the previous situation, the crucial fact is to decide if, indeed, $u$ and $y$ are related to each other through the state equation, i.e. $y\equiv y_u$
\begin{equation}\label{limiteestado}
-\dv[A(\nabla y)]=f(\bbx, u)\hbox{ in }\Omega.
\end{equation}
This time however, we cannot resort to the same arguments as in Section \ref{variacional} as we are dealing with a non-variational PDE. Yet, because, the control variable $u$ does not occur in the main part of the equation, the usual arguments based on monotonicity still can be used and applied. This set of ideas is identified as the method of Browder and Minty (\cite{evans}). 

\begin{theorem}\label{existencia2}
Optimization problem \eqref{probnonvar} under state equation \eqref{estado2} admits optimal solutions provided hypotheses \eqref{rhs}-\eqref{monotonia3} hold.
\end{theorem}
\begin{proof}
We start from the non-negativeness condition
$$
\int_\Omega(A(\nabla y_j)-A(\nabla v))\cdot(\nabla y_j-\nabla v)\,d\bbx\ge0,
$$
which is valid for an arbitrary $v\in H^1_0(\Omega)$. For the first term
$$
\int_\Omega A(\nabla y_j)\cdot\nabla y_j\,d\bbx,
$$
we know, because of \eqref{ecuacion}, that it is equal to
$$
\int_\Omega f(\bbx, u_j(\bbx))y_j(\bbx)\,d\bbx,
$$
and hence
$$
\int_\Omega[f(\bbx, u_j)y_j-A(\nabla y_j)\nabla v-A(\nabla v)\cdot(\nabla y_j-\nabla v)]\,d\bbx\ge0.
$$
We would like to take limits in $j$ in this inequality. This forces us to write 
$$
A(\nabla y_j)\rightharpoonup\overline A\hbox{ in }L^2(\Omega; \R^N)
$$ 
for some function $\overline A$ which, in general, is not $A(\nabla y)$. This is in fact the heart of the difficulty. We will come back to this crucial issue in the next section where we will be most interested in understanding a way to tame these non-linear weak limits. Because the product $f(\bbx, u_j)y_j$ converges strongly to $f(\bbx, u)y$ since both factors do, we have
$$
\int_\Omega[f(\bbx, u)y-\overline A\cdot\nabla v-A(\nabla v)\cdot(\nabla y-\nabla v)]\,d\bbx\ge0.
$$
On the other hand, taking limits in \eqref{ecuacion} leads to
\begin{equation}\label{limite}
\int_\Omega \overline A\cdot\nabla w\,d\bbx=\int_\Omega f(\bbx, u) w\,d\bbx
\end{equation}
for every $w\in H^1_0(\Omega)$. In particular, using this identity for $w=y$, we can also write our resulting inequality in the form
$$
\int_\Omega(\overline A-A(\nabla v)\cdot(\nabla y-\nabla v)\,d\bbx\ge0,
$$
for all $v\in H^1_0(\Omega)$. If we now take $y-v=\lambda w$ for a positive scalar $\lambda$, 
$$
\int_\Omega(\overline A-A(\nabla y-\lambda\nabla w))\cdot\nabla w\,d\bbx\ge0.
$$
If we take $\lambda\to0$, we find
$$
\int_\Omega(\overline A-A(\nabla y))\cdot\nabla w\,d\bbx\ge0.
$$
The arbitrariness of $w\in H^1_0(\Omega)$ implies that, in fact, 
$$
\int_\Omega(\overline A-A(\nabla y))\cdot\nabla w\,d\bbx=0,
$$
and, then, recalling \eqref{limite}, we conclude
$$
\int_\Omega A(\nabla y)\cdot\nabla w\,d\bbx=\int_\Omega f(\bbx, u)w\,d\bbx,
$$
for every $w\in H^1_0(\Omega)$. This shows that \eqref{limiteestado} is correct. 

The rest of the proof proceeds in a standard way just like that of Theorem \ref{existencia}. Note that we could allow a more general integrand $F$ depending in a convex way in $\nabla y_u$. 
\end{proof}

The more general problem that is the result of changing the term
$$
\frac M2\int_\Omega|\nabla u(\bbx)|^2\,d\bbx
$$
by 
$$
\frac M2\int_\Omega|u(\bbx)|^2\,d\bbx
$$
in the functional cost has some important consequences concerning existence. As a matter of fact, the convergence $u_j$ to some feasible $u$ is no longer strong, and one has to be contented with only weak convergence. In such a situation, the limit of the right-hand side $f(\bbx, u_j)$ is no longer $f(\bbx, u)$, not even in a weak sense. In general, there is no way to recover the existence of optimal solutions, unless $f(\bbx, u)$ is linear in $u$. The proof of Theorem \ref{existencia2} is still valid but the limit of $f(\bbx, u_j)$ needs to be  conveniently identified. This can only be done through the Young measure $\nu=\{\nu_\bbx\}_{\bbx\in\Omega}$ associated with (a suitable subsequence of) a minimizing sequence $\{u_j\}$. There is a reach tradition of the use of Young measures in optimal control problems, precisely to furnish tools to deal with those problems for which non-existence, as a result of some non-convexity, is a fact (\cite{youngp}, \cite{young}). These analytical tools have been used to provide descriptions of relaxed versions of optimal control problems without solutions in a variety of scenarios. Check \cite{roubicek} for an account of a more recent, quite sophisticated treatment of these issues. 

Since the use of Young measures associated with sequences of feasible controls, at least as existence and relaxation is concerned, is by now well understood, we will not insist on this point in this section. Rather, we would like to see how Young measures corresponding to sequences of associated states $\{y_j\equiv y_{u_j}\}$ can be utilized to deal with main non-linearities in the state equation.

\section{Non-linear state equations}
In the final step, we would like to go from a state equation
$$
-\dv[A(\nabla y(\bbx))]=f(\bbx, u(\bbx))\hbox{ in }\Omega
$$
to 
$$
-\dv[A(\nabla y(\bbx))]+a(\nabla y)+by=f(\bbx, u(\bbx))\hbox{ in }\Omega.
$$
Note how the presence of the non-divergence term $a(\nabla y)$ spoils, in general, the strategy we have used in the previous two sections. As a matter of fact, the nature of optimal control problems under such state equations is rather different as we will shortly see. Since, on the other hand, the role played by the monotone, divergence term $A(\nabla y)$ is similar to what we have seen in the last section, for the sake of simplicity, we will only examine the case 
$$
A(\nabla y)=\nabla y, \quad -\dv[A(\nabla y)]=-\Delta y,
$$
so that our state equation will be taken to be
$$
-\Delta y+a(\nabla y)+by=f(u)\hbox{ in }\Omega,
$$
where
$$
|a(\by)|\le C(|\by|+1),\quad |f(u)|\le C(|u|+1),
$$
for some constant $C>0$. Both $a$ and $f$ are assumed to be continuous. The final term $by$ is there to ensure the existence of a solution of the state equation. If, in addition, $a$ is Lipschitz, the solution is unique.
\begin{lemma}\label{exist}
Suppose $a:\R^N\to\R$ is continuous and 
$$
|a(\by)|\le C(|\by|+1),\quad C>0.
$$
If $b>0$ is sufficiently large, then there are solutions of 
\begin{equation}\label{ecuestado}
-\Delta y+a(\nabla y)+by=F\hbox{ in }\Omega, \quad y\in H^2(\Omega)\cap H^1_0(\Omega),
\end{equation}
for $F\in L^2(\Omega)$. If, in addition, $a$ is Lipschitz, 
the solution is unique (for $b$ even larger if necessary). 
\end{lemma}
\begin{proof}
The existence for this kind of quasilinear equations, for $b$ sufficiently large, is established through fixed-point techniques as an application of Schaefer's Fixed Point Theorem. Check \cite{evans}, Section 9.2.2, or \cite{gilbarg} for more elaborate situations. 

The uniqueness part is elementary. If $y_i$, $i=1, 2$, are two solutions, multiplying the difference of the equations corresponding to $y_1$ and $y_2$, by the difference $y_1-y_2$, and performing a standard integration by parts, we find
\begin{equation}\label{ecuaciondebil}
\int_\Omega [|\nabla y_1-\nabla y_2|^2+(a(\nabla y_1)-a(\nabla y_2))(y_1-y_2)+b(y_1-y_2)^2]\,d\bbx=0.
\end{equation}
If $L$ is the Lipschitz constant of the function $a$, then 
$$
|(a(\nabla y_1)-a(\nabla y_2))(y_1-y_2)|\le L|\nabla y_1-\nabla y_2|\,|y_1-y_2|,
$$
and for $\epsilon>0$, we can write
$$
|(a(\nabla y_1)-a(\nabla y_2))(y_1-y_2)|\le \frac L2 \epsilon^2|\nabla y_1-\nabla y_2|^2+\frac L2\frac1{\epsilon^2}|y_1-y_2|^2.
$$
If we take this inequality back to identity \eqref{ecuaciondebil}, we have
$$
\int_\Omega\left[\left(1-\frac L2\epsilon^2\right)|\nabla y_1-\nabla y_2|^2+\left(b-\frac L2\frac1{\epsilon^2}\right)|y_1-y_2|^2\right]\,d\bbx\le0.
$$
If we take set $\epsilon^2=2/L$, and take $b$ sufficiently large, then
$$
0\le \left(b-\frac{L^2}4\right)\int_\Omega(y_1-y_2)^2\,d\bbx\le0
$$
implies $y_1\equiv y_2$. 
\end{proof}

Once we have clarified the nature of the state equation \eqref{ecuestado}, and the conditions under which it admits a unique solution, we focus on the optimal control problem
\begin{equation}\label{nolineal}
\hbox{Minimize in }u\in H^1(\Omega):\quad E(u)=\int_\Omega F(y_u(\bbx))\,d\bbx+\frac M2\int_\Omega|\nabla u(\bbx)|^2\,d\bbx
\end{equation}
under
\begin{equation}\label{estadonolineal}
-\Delta y_u+a(\nabla y_u)+by_u=f(u)\hbox{ in }\Omega,\quad y_u\in H^2(\Omega)\cap H^1_0(\Omega),
\end{equation}
for positive constant $M$, continuous, uniformly-bounded function $f$, and Lipschitz $a$. Without loss of generality, we can assume that $a(\bcero)=0$, because an arbitrary constant can be incorporated in the function $f$. In this way, we can assume that
\begin{equation}\label{lips}
|a(\by)|\le C|\by|,\quad C>0.
\end{equation}
The function $f$ has been assumed uniformly bounded to avoid some technical difficulties. 
The parameter $b$ is chosen sufficiently large depending on $a$, $f$, and $D$, to ensure existence and uniqueness of the associated state $y_u$ for each feasible $u$, according to Lemma \ref{exist}. We suppose that these ingredients are fixed once and for all so that the optimal control problem is well-defined. 

To realize the difficulties to show existence of optimal solutions in this more complicated framework, as compared to the other situations treated earlier, suppose $\{u_j\}$ is a minimizing sequence, which is bounded in $H^1(\Omega)$, and converging to some feasible $u$ strong in $L^2(\Omega)$, weak in $H^1(\Omega)$, and pointwise. Let $\{y_j\}$ be the sequence of corresponding states
\begin{equation}\label{estadosj}
-\Delta y_j+a(\nabla y_j)+by_j=f(u_j)\hbox{ in }\Omega,\quad y_j\in H^2(\Omega)\cap H^1_0(\Omega).
\end{equation}
Similar calculations as the ones utilized in the uniqueness part of Lemma \ref{exist} lead us to conclude that $\{y_j\}$ is uniformly bounded in $H^1(\Omega)$, and hence it converges to some $y$ weakly in this space. As in previous situations, the whole issue is to ensure that this limit $y$ is indeed the associated state corresponding to the above limit $u$. The presence of the non-linear term $a(\nabla y_j)$ outside the divergence part of the equation 
makes the success of the previous convexity/monotonicity arguments impossible so that there is, apparently, no way to prove that
$$
-\Delta y+a(\nabla y)+by=f(u)\hbox{ in }\Omega.
$$
If we multiply \eqref{estadosj} by a test function $z\in H^1_0(\Omega)$, and integrate by parts in the first term, we see that
$$
\int_\Omega[\nabla y_j\cdot\nabla z+a(\nabla y_j)z+by_jz]\,d\bbx=\int_\Omega f(u_j)z\,d\bbx.
$$
When taking limits in $j$, we can identify the limit for all the terms except for the product $a(\nabla y_j)z$. Indeed, if $a(\nabla y_j)\rightharpoonup\overline a$ weakly in $L^2(\Omega)$, then
$$
\int_\Omega[\nabla y\cdot\nabla z+\overline a z+byz]\,d\bbx=\int_\Omega f(u)z\,d\bbx.
$$
In general, $\overline a$ is not $a(\nabla y)$, as we well know. There is no way to resolve the situation, except by the use of some tool that permits to express and manipulate the weak limits of non-linear quantities of the form $\{a(\nabla y_j)\}$. This is the whole point of Young measures, and measure-valued solutions of PDEs. We need to incorporate these special limit PDEs into the problem to ensure existence of optimal solutions, as is usual in relaxed formulations of optimization problems. 

Before formulating in precise terms the generalized optimal control problem, let us recall two fundamental well-known results taken from the theory of Young measures (\cite{balder}, \cite{ball}, \cite{muller}, \cite{rindler}, \cite{roubicek}, \cite{pedregalbook}, \cite{valadier}). The statement for both of them has been adapted to the situation in this work. Much more general versions are possible. 
The first one is the typical result that stresses the fundamental property of the Young measures corresponding to a sequence of functions. It is capable of reproducing weak limits of compositions with arbitrary non-linear quantities. 
\begin{theorem}\label{basico}
Let $\Omega\subset\R^N$ be a regular, bounded domain, and let $\by_j:\Omega\to\R^N$ be a sequence of measurable fields such that
$$
\sup_j\int_\Omega |\by_j(\bbx)|^2\,d\bbx<\infty.
$$
There is a subsequence, not relabeled, and a family of probability measures $\nu=\{\nu_\bbx\}_{\bbx\in\Omega}$ (the associated Young measure) supported in the target space $\R^N$, such that
$$
\int_\Omega\int_{\R^N}|\lambda|^2\,d\nu_\bbx(\lambda)\,d\bbx<\infty,
$$
and with the property that whenever the sequence $\{\psi(\bbx, \by_j(\bbx))\}$ is weakly convergent in $L^1(\Omega)$ for an arbitrary Carath\'eodory integrand $\psi(\bbx, \by)$, the weak limit is the function
\begin{equation}\label{representacion}
\overline\psi(\bbx)=\int_{\R^N}\psi(\bbx, \lambda)\,d\nu_\bbx(\lambda),\quad \psi(\bbx, \by_j(\bbx))\rightharpoonup \overline\psi(\bbx).
\end{equation}
\end{theorem}
The important point is that the Young measure is determined by the sequence $\{\by_j\}$ but it is independent of the function $\psi$.

The second one characterizes the family of probability measures that can occur as Young measures of arbitrary sequences $\{\nabla y_j\}$ for uniformly bounded sequences $\{y_j\}$ in $H^1(\Omega)$. This particular statement is adapted directly from \cite{pedregalbook} (Theorem 8.7). 
\begin{theorem}\label{importante}
Let $\nu=\{\nu_\bbx\}_{\bbx\in\Omega}$ be a family of probability measures supported in $\R^N$ and depending measurably on $\bbx$, where $\Omega\subset\R^N$ is a regular, bounded domain. The necessary and sufficient conditions to find a bounded sequence of functions $\{y_j\}$ in $H^1_0(\Omega)$ with $\{|\nabla y_j|^2\}$, weakly convergent in $L^1(\Omega)$, and with associated Young measure $\nu$ (according to the previous statement) are
\begin{gather}
\nabla y(\bbx)=\int_{\R^N}\lambda\,d\nu_\bbx(\lambda),\quad y\in H^1_0(\Omega),\label{primero}\\
\int_\Omega\int_{\R^N}|\lambda|^2\,d\nu_\bbx(\lambda)\,d\bbx<\infty.\label{segundo}
\end{gather}
\end{theorem}
As a result of this last theorem, we will designate by $PH^1_0(\Omega)$ the set of families of probability measures supported in $\R^N$ verifying \eqref{primero}-\eqref{segundo}. 

An important consequence of these two fundamental results is worth stating by itself.
\begin{corollary}\label{modificar}
Let $\{u_j\}$ be a uniformly-bounded sequence in $H^1(\Omega)$, with $\nu$, the Young measure associated with the sequence of their gradients $\{\nabla u_j\}$. There is another sequence $\{\tilde u_j\}$, uniformly-bounded in $H^1(\Omega)$ with $\{|\nabla \tilde u_j|^2\}$, converging weakly in $L^1(\Omega)$, generating the same Young measure $\nu$. 
\end{corollary}
The main issue here is that for the initial sequence $\{\nabla u_j\}$, it may not be true that $\{|\nabla u_j|^2\}$ is weakly convergent in $L^1(\Omega)$; yet, it can be modified slightly to $\{\nabla\tilde u_j\}$, without changing the underlying Young measure, in such a way that $\{|\nabla\tilde u_j|^2\}$ is indeed weakly convergent in $L^1(\Omega)$.

The following corollary is directly adapted for our purposes for the study of our optimal control problem \eqref{nolineal}.
\begin{corollary}\label{aqui}
Let $\nu\in PH^1_0(\Omega)$, and $\{\nabla y_j\}\subset H^1_0(\Omega)$, a generating sequence (with Young measure $\nu$). If $a(\by):\R^N\to\R$ is such that 
$$
|a(\by)|\le C(|\by|^\alpha+1),\quad \alpha<2,
$$
then
\begin{equation}\label{repro}
a(\nabla y_j)\rightharpoonup \overline a(\bbx)\equiv\int_{\R^N}a(\lambda)\,d\nu_\bbx(\lambda)\hbox{ in } L^1(\Omega).
\end{equation}
\end{corollary}
\begin{proof}
By our result above, the sequence of squares $\{|\nabla y_j|^2\}$ is uniformly bounded. If the exponent $\alpha<2$, the sequence of $\alpha$-powers $\{|\nabla y_j|^\alpha\}$ is weakly convergent in $L^1(\Omega)$ (this is elementary). By the reproducing property \eqref{representacion} of the Young measure $\nu$, we immediately find \eqref{repro}. 
\end{proof}

We are now in a position to start looking for a relaxed formulation of our optimal control problem, once we have recalled the above fundamental tools and results we are in need of. 
What might look like a good candidate for relaxation is, as a matter of fact, just a first step. We will see that a second stage is necessary. 
We formulate an intermediate optimal control problem precisely in the following terms. 
Consider the following generalized optimal control problem
\begin{equation}\label{relajacion}
\hbox{Minimize in }u\in H^1(\Omega):\quad E(u)=\int_\Omega F(y_u(\bbx))\,d\bbx+\frac M2\int_\Omega|\nabla u(\bbx)|^2\,d\bbx
\end{equation}
where 
$$
\nu_u=\{\nu_{u, \bbx}\}_{\bbx\in\Omega}\in PH^1_0(\Omega),
$$
is a measure-valued solution of the state equation \eqref{estadonolineal}, i.e.
\begin{gather}
-\Delta y_u+\overline a+by_u=f(u)\hbox{ in }\Omega,\label{solmed}\\
\nabla y_u(\bbx)=\int_{\R^N}\lambda\,d\nu_{u, \bbx}(\lambda),\quad \overline a(\bbx)=\int_{\R^N}a(\lambda)\,d\nu_{u, \bbx}(\lambda),\quad y_u\in H^2(\Omega)\cap H^1_0(\Omega).\nonumber
\end{gather}
\begin{proposition}
Problem \eqref{relajacion} is a sub-relaxation of problem \eqref{nolineal} in the sense $\overline m\le m$ if $\overline m$ and $m$ are the respective infima of both problems.
\end{proposition}
\begin{proof}
Under the identification, as it is usual when Young measures are used in optimization problems, 
\begin{equation}\label{delta}
\nu_{u, \bbx}=\delta_{\nabla y_u(\bbx)},\quad -\Delta y_u+a(\nabla y_u)+by_u=f(u)\hbox{ in }\Omega,\quad y_u\in H^2(\Omega)\cap H^1_0(\Omega),
\end{equation}
it is easy to realize that $E(u)$, computed in \eqref{nolineal} or calculated in \eqref{relajacion} under \eqref{delta}, is the same number. This immediately implies that $\overline m\le m$. 
\end{proof}
The fundamental point of a full relaxation is to show the equality $\overline m=m$, and to identify a true minimizer for \eqref{relajacion}. Suppose we pretend to prove that this is so, and let $\{u_j\}$ be a minimizing sequence of the original problem, with $\{y_j\}$ the sequence of their corresponding states. We are certain that $E(u_j)\to m$, the infimum of the problem, and that
$$
-\Delta y_j+a(\nabla y_j)+by_j=f(u_j)\hbox{ in }\Omega.
$$ 
We claim that the sequence $\{|\nabla y_j|^2\}$ is uniformly bounded in $L^1(\Omega)$. To check this, we perform some elementary calculations similar to the ones in the proof of Lemma \ref{exist}. Multiply the previous state equation by $y_j$, integrate by parts in the first term, and complete squares in the form
$$
\int_\Omega\left[\left(1-\frac{C^2}{4b}\right)|\nabla y_j|^2+b\left(y_j-\frac C{2b}y_j^2\right)^2\right]\,d\bbx=\int_\Omega f(u_j)y_j\,d\bbx.
$$
From here, we deduce in a standard way that
$$
\left(1-\frac{C^2}{4b}\right)\|\nabla y_j\|^2_{L^2(\Omega)}\le \|f(u_j)\|_{L^2(\Omega)}\,\|y_j\|_{L^2(\Omega)},
$$
and, through Poincar\'e's inequality, the uniform bound on $\{u_j\}$, and the bound on $f$, we can conclude that
$$
\left(1-\frac{C^2}{4b}\right)\|\nabla y_j\|_{L^2(\Omega)}\le M
$$
for some fixed $M$, depending on $f$, but independent of $j$. Note that we are allowing the constant $b$ to be taken as large as necessary, depending on the various ingredients of the problem. In particular, it should be true that 
$$
\left(1-\frac{C^2}{4b}\right)>0.
$$
Recall that $C$ is the Lipschitz constant in \eqref{lips}. 

According to Theorem \ref{basico}, for a non-relabeled subsequence of $\{\nabla y_j\}$, there is an associated Young measure $\nu=\{\nu_{\bbx}\}_{\bbx\in\Omega}$ with 
\begin{equation}\label{cuadratico}
\int_\Omega\int_{\R^N}|\lambda|^2\,d\nu_\bbx(\lambda)\,d\bbx<\infty
\end{equation}
and complying with the representation property \eqref{representacion}. For functions with linear growth, \eqref{representacion} is correct under \eqref{cuadratico}; in particular
$$
\nabla y_j\rightharpoonup \int_{\R^N}\lambda\,d\nu_\bbx(\lambda)=\nabla y(\bbx),
$$
and
$$
a(\nabla y_j)\rightharpoonup \int_{\R^N}a(\lambda)\,d\nu_\bbx(\lambda)=\overline a(\bbx).
$$
These conclusions imply that $\nu\in PH^1_0(\Omega)$ is a measure-valued solution of \eqref{solmed}. Note that $f(u_j)\to f(u)$ because the convergence $u_j$ to $u$ is strong, as we have a uniform bound on the size of $\{\nabla u_j\}$. Finally, it is easy to find that 
$$
E(u)\le \liminf_{j\to\infty}E(u_j),\quad E(u)\equiv E(u, \nu),
$$
because $y_j\to y$ strong in $L^2(\Omega)$ and point-wise, as well as 
$$
\int_\Omega |\nabla u|^2\,d\bbx\le\liminf_{j\to\infty}\int_\Omega|\nabla u_j|^2\,d\bbx,
$$
through convexity and weak lower semicontinuity. Therefore
$$
\overline m\le E(u, \nu)\le\liminf_{j\to\infty}E(u_j)=m
$$
but this is not yet enough to conclude $\overline m=m$. This equality would force
$$
\int_\Omega |\nabla u|^2\,d\bbx=\liminf_{j\to\infty}\int_\Omega|\nabla u_j|^2\,d\bbx,
$$
and this in turn, together with the weak convergence $\nabla u_j\rightharpoonup \nabla u$ would imply the strong convergence $u_j\to u$ in $H^1(\Omega)$, which, in general, would be hard to achieve. There is, therefore, an additional step to take to reach a full relaxation. It involves to describe the behavior of minimizing sequences of feasible controls through their associated Young measures as well.

Consider the optimal control problem
\begin{equation}\label{relaxation}
\hbox{Minimize in }\mu:\quad E(\mu)=\int_\Omega F(y_\mu(\bbx))\,d\bbx+\frac M2\int_\Omega\int_{\R^N}|\lambda|^2\,d\mu_{\bbx}(\lambda)\,d\bbx
\end{equation}
where 
$$
\mu=\{\mu_\bbx\}_{\bbx\in\Omega}\in PH^1(\Omega),
$$
and
$$
\nu=\{\nu_{\mu, \bbx}\}_{\bbx\in\Omega}\in PH^1_0(\Omega),
$$
is a measure-valued solution of the state equation \eqref{estadonolineal}, i.e.
\begin{equation}\label{solmeddos}
-\Delta y_\mu+\overline a+by_\mu=f(u)\hbox{ in }\Omega,
\end{equation}
with
\begin{gather}
\nabla u(\bbx)=\int_{\R^N}\lambda\,d\mu_\bbx(\lambda),\quad u\in H^1(\Omega),\nonumber\\
\nabla y_\mu(\bbx)=\int_{\R^N}\lambda\,d\nu_{\mu, \bbx}(\lambda),\quad \overline a(\bbx)=\int_{\R^N}a(\lambda)\,d\nu_{\mu, \bbx}(\lambda),\quad y_\mu\in H^2(\Omega)\cap H^1_0(\Omega).\nonumber
\end{gather}
As can be figured out, $PH^1(\Omega)$ stands for the class of Young measures complying with \eqref{primero}-\eqref{segundo}, replacing the condition $u\in H^1_0(\Omega)$ by simply $u\in H^1(\Omega)$. 

\begin{theorem}\label{relajaciont}
Problem \eqref{relaxation} is a full relaxation of problem \eqref{nolineal} in the sense:
\begin{enumerate}
\item If $\overline m$ and $m$ are the respective infima of both problems then $\overline m=m$.
\item $\overline m$ is attained: there is a feasible minimizer $\mu$ for \eqref{relaxation}.
\end{enumerate}
\end{theorem}
\begin{proof}
The proof has essentially been indicated in our above discussion. On the one hand, by considering trivial families of Dirac delta measures, it is easy to realize that the original problem \eqref{nolineal} is embedded into \eqref{relaxation}, and hence, $\overline m\le m$. 

On the other, let $\mu\in PH^1(\Omega)$ be a feasible element for \eqref{relaxation} with
\begin{equation}\label{primerm}
\nabla u(\bbx)=\int_{\R^N}\lambda\,d\mu_\bbx(\lambda),\quad u\in H^1(\Omega).
\end{equation}
According to Theorem \ref{importante}, there is a sequence of feasible controls $\{u_j\}$ in such a way that 
\begin{equation}\label{convergencia}
\int_\Omega |\nabla u_j|^2\,d\bbx\to\int_\Omega\int_{\R^N}|\lambda|^2\,d\mu_\bbx(\lambda)\,d\bbx,
\end{equation}
and $u_j$ converging to $u$ weakly in $H^1(\Omega)$, and strongly in $L^2(\Omega)$. 


Suppose now that $\{y_j\}$ is the sequence of states corresponding to $\{u_j\}$ under \eqref{convergencia}. The sequence of gradients $\{\nabla y_j\}$ is uniformly bounded in $L^2(\Omega)$, just as we have argued before. Let, according to Theorem \ref{basico}, $\nu$ be its respective Young measures (again for a suitable, non-relabeled subsequence). From
$$
-\Delta y_j+a(\nabla y_j)+by_j=f(u_j)\hbox{ in }\Omega,
$$
one can easily deduce \eqref{solmeddos}. Recall that $u_j\to u$ point-wise and in $L^2(\Omega)$, so that $f(u_j)\to f(u)$. In this way we see that $\nu\equiv\nu_\mu$, $y_j\to y_\mu$, and \eqref{convergencia} guarantees that
$$
E(\mu)=\lim_{j\to\infty} E(u_j).
$$
The arbitrariness of $\mu$ implies that indeed $\overline m=m$. 

If, further, we take a minimizing sequence $\{u_j\}$ for our original problem \eqref{nolineal}, and we let $\mu$ be its corresponding Young measure, by Corollary \ref{modificar}, possibly changing slightly $\{u_j\}$, we can take for granted that $\{|\nabla u_j|^2\}$ is weakly convergent so that \eqref{convergencia} holds, and then conclude that
$$
\overline m\le E(\mu)=\lim_{j\to\infty} E(u_j)=m=\overline m.
$$
$\mu$ is indeed a minimizer of \eqref{relaxation}. 
\end{proof}

\end{document}